\documentclass[letterpaper, 10pt, conference]{ieeeconf}  

\IEEEoverridecommandlockouts                              
\usepackage{amsmath,amssymb}
\usepackage{url}
\usepackage{times} 
\usepackage{algorithm}
\usepackage{algpseudocode}
\let\labelindent\relax
\usepackage{enumitem}
\usepackage{optidef}
\usepackage{hyperref}
\usepackage{svg}
\hypersetup{
    colorlinks=true,
    linkcolor=blue,
    filecolor=magenta,      
    urlcolor=cyan,
    pdftitle={Overleaf Example},
    pdfpagemode=FullScreen,
    }
\usepackage{subcaption}
\usepackage{booktabs}
\usepackage{url}

\newtheorem{theorem}{Theorem}
\newtheorem{proposition}[theorem]{Proposition}%
\newtheorem{remark}{Remark}%
\usepackage{tikz}
\newtheorem{definition}{Definition}%

\usepackage{graphicx,graphics,epsfig,color}
\usepackage{multicol}
\usepackage{float}
\usepackage{tikz}
\usepackage{comment}

\newcommand{\Pre}{\textrm{Pre}}

\newcommand{\Int}{\textrm{Int}}
\newcommand{\cl}{\textrm{cl}}
\newcommand{\ch}{\textrm{ch}}
\newcommand{\ex}{\textrm{ex}}
\newcommand{\af}{\textrm{Aff}}
\newcommand{\ri}{\textrm{ri}}
\newcommand{\rdelta}{\textrm{r}\partial}
\newcommand{\Single}{\textrm{Single}}

\newif\ifitsdraft
\def\itsdraft{\global\itsdrafttrue}

\itsdraft
\newif\ifitssubmit

\newtheorem{assumption}{Assumption}

\newtheorem{lemma}{Lemma}

\newtheorem{corollary}{Corollary}

\usepackage{graphicx}
\usepackage{textcomp}
\usepackage{xcolor}

\title{\LARGE \bf
A Trajectory-Based Approach to Controlled Invariance and Recursively Feasible MPC}

\author{Emmanuel J. Wafo Wembe and Adnane Saoud
\thanks{Emmanuel J. Wafo Wembe and Adnane Saoud are with The college of computing,
        Mohammed VI Polytechnic University,Ben Guerir, Morocco
        {\tt\small \{emmanueljunior.wafowembe, adnane.saoud\}@um6p.ma}}
}

\begin{document}
\maketitle
\thispagestyle{empty}
\pagestyle{empty}
\begin{abstract}
    In this paper, we revisit the computation of controlled invariant sets for linear discrete-time systems through a trajectory-based viewpoint. We begin by introducing the notion of convex feasible points, which provides a new characterization of controlled invariance using finitely long state trajectories. We further show that combining this notion with the classical backward fixed-point algorithm allows for the computation of the maximal controlled invariant set. Building on these results, we propose a model predictive control (MPC) scheme that guarantees recursive feasibility without relying on precomputed terminal sets. Finally, we formulate the search for convex feasible points as an optimization problem, yielding a practical computational method for constructing controlled invariant sets. The effectiveness of the approach is illustrated through numerical examples.
\end{abstract}
\section{Introduction}
Set invariance is crucial in  constrained control and formal verification~\cite{blanchini_set-theoretic_2008}. For discrete-time linear time-invariant (LTI) systems, a controlled invariant (CI) set provides a region where admissible control inputs ensure state constraint satisfaction over an infinite horizon. Classical computation of these sets typically relies on recursive predecessor operations, leading to fixed-point algorithms that determine the \emph{maximal} CI set~\cite{kolmanovsky_gilbert_1998}. However, these backward-propagation methods face a "curse of dimensionality" due to their reliance on polyhedral operations,such as projections and Minkowski sums, which exhibit exponential computational complexity relative to the state dimension.

To mitigate this, recent works have explored LMI conditions~\cite{BRIAO20219741}, Newton-type accelerations~\cite{Chenyuan_tahir_2019}, and implicit representations~\cite{anevlavs-lui-ozay_2024}. While effective, these remains tied to recursive set-propagation logic. Motivated by scalability, we propose a shift toward a trajectory-based perspective, certifying invariance via the convex hull of finite-length trajectories. By introducing \emph{convex feasible points}, this framework encodes infinite-time invariance into a single-shot optimization program, shifting complexity from the state dimension to the prediction horizon.

This approach bridges the gap between finite-horizon trajectory optimization and infinite-horizon set theory. Unlike the implicit polyhedral methods in~\cite{anevlavs-lui-ozay_2024}, our method bypasses explicit set propagation entirely. This makes the framework a practical alternative for high-dimensional systems where classical tools, such as the MPT toolbox, fail to converge. Furthermore, we demonstrate that by integrating this trajectory-based certificate into a backward-propagation scheme, one can iteratively approximate the maximal CI set with high precision.

\textbf{Related works.}  
Classical invariant-set computation for LTI systems is established in \cite{blanchini_set-theoretic_2008, kolmanovsky_gilbert_1998}, typically using recursive backward-reachable algorithms \cite{7862214, Chenyuan_tahir_2019}. While recent works introduce low-complexity approximations \cite{Tahir_jaimoukha_2015, BRIAO20219741} and implicit representations \cite{anevlavs-lui-ozay_2024}, these remain fundamentally set-based. Trajectory-based characterizations have successfully bypassed set-propagation for monotone systems \cite{9992752, saoud_characterization_2024, WAFOWEMBE2024135}. Building on this, we develop a trajectory-based framework for general LTI systems, encoding controlled invariance as a single feasibility program that complements classical recursive techniques.

Backward fixed-point methods require repeated projections, which are costly in high dimensions. In contrast, the proposed approach computes a candidate invariant set in a \emph{single} step and provides trajectory-based certificates, enabling efficient approximations with good convergence.

The primary contributions of this work are:
\begin{itemize}
    \item A theoretical characterization of controlled invariance based on the convex hull of finite-length trajectories.
    \item A convex feasibility framework that encodes this certificate as a single optimization problem, ensuring computational tractability in high dimensions.
    \item A recursively feasible MPC scheme that utilizes trajectory-induced sets as "on-the-fly" terminal constraints, eliminating the need for offline computation.
    \item A comparative analysis with state-of-the-art methods
\end{itemize}

The paper is organized as follows: Section~\ref{sec:prelim} introduces notation and the system class. Section~\ref{sec:traj} develops the trajectory-based framework, while Section~\ref{sec:MPC} applies it to MPC. Section~\ref{sec:computation} describes the optimization, and Section~\ref{sec:Examples} provides numerical validation.
\ifitssubmit
Due to space, proofs are in the online version of this \href{https://arxiv.org/abs/2604.07225}{paper}.
\fi
\section{Preliminaries}
\label{sec:prelim}
\subsection{Notations}
The symbols $\mathbb{N}$, $ \mathbb{N}_{>0} $, $\mathbb{R}$, and $\mathbb{R}_{>0}$ denote the set of positive integers, non-negative integers, real numbers, and non-negative real numbers, respectively.  Given a nonempty set $K$, $\Int(K)$ denotes its interior, $\cl(K)$ denotes its closure, $\partial K$ denotes its boundary, and $\overline{K}$ its complement. For a set $K$, the operator $\Single(K)$ randomly selects a unique element from the set $K$. The Euclidean norm is denoted by $\|.\|$. For $x \in \mathbb{R}^n$ and for $\varepsilon >0$, $\mathcal{B}_{\varepsilon}(x)=\{z \in \mathbb{R}^n \mid \|z-x\| \leq \varepsilon \}$. For $x,y \in \mathbb{R}^n$, we denote by $[x,y]$, $]x,y[$ the set $\{tx+(1-t)y, t\in [0,1]\}$ and $\{tx+(1-t)y, t\in ]0,1[\}$ respectively.  For a matrix $A \in \mathbb{R}^{n\times m}$, we defines $\|A\| = \sup\limits_{x \neq 0} \frac{\|Ax\|}{\|x\|}$. For any sets $A,B$ the Minkowski sum is defined as $A\oplus B=\{a+b, a\in A, b \in B\}$ and $A\ominus B=\{a-b, a\in A, b \in B\}$. 


We denote the convex hull of a set of points as $\ch(\{x_i\})$. The relative interior of a set $C$ is denoted $\text{ri}(C)$.

The following result introduce a characterization of the relative interior of the convex hull of a set of point. 
\begin{lemma}
\label{lem:Int_char}
    Consider any points $x_0, x_1, \dots, x_n \in \mathbb{R}^d$. Let $x \in \mathbb{R}^d$ , we have $x \in \ri(\ch(x_0, x_1, \dots, x_n))$ if and only if there exists $ (\lambda_0, \lambda_1, \dots ,  \lambda_n) > 0 $ such that $\sum\limits_{i =1}^n \lambda_i = 1 $ and $\sum\limits_{i =1}^n \lambda_i x_i = x$.
\end{lemma}
\subsection{Linear discrete-time  control systems}

In this paper, we consider the class of linear discrete-time control systems $\Sigma$ of the form:
\begin{equation}
\label{dis_sys}
x^+  = A x + B u 
\end{equation}
where $x \in \mathcal{X} \subseteq \mathbb{R}^n$ is the state, $u \in \mathcal{U} \subseteq \mathbb{R}^m$ is the control input. The trajectories of (\ref{dis_sys}) are denoted by $\Phi(.,x_0,\mathbf{u})$ where $\Phi(t,x_0,\mathbf{u})$ is the state reached at time $t \in \mathbb{N}_{\geq 0}$ from the initial state $x_0$ under the control input $\mathbf{u}:\mathbb{N}_{\geq 0} \rightarrow \mathcal{U}$.

When the control inputs of system $\Sigma$ in (\ref{dis_sys}) are generated by a state-feedback controller $\kappa:\mathcal{X} \rightarrow \mathcal{U}$,  the dynamics of the closed-loop system is given by
\begin{equation}
\label{eqn:syst_cl}
    x^+ = A x + B \kappa(x) 
\end{equation}
The trajectories of (\ref{eqn:syst_cl}) are denoted by $\Phi_{\kappa}(.,x_0)$ where $\Phi_{\kappa}(t,x_0)$ is the state reached at time $t \in \mathbb{N}_{\geq 0}$ from the initial state $x_0$. 

\subsection{Controlled Invariance}

Over the years, controlled invariance has been a cornerstone of many theoretical and applied research efforts \cite{blanchini_set-theoretic_2008,rakovic_parameterized_2010}. We start by recalling the concept of controlled invariants.

\begin{definition}
\label{def:RCI}
    Consider the system $\Sigma$ in (\ref{dis_sys})  and let $X \subseteq \mathcal{X}$ and $U \subseteq \mathcal{U}$  be the constraint sets on the states, and inputs, respectively. A set $S$ is a controlled invariant for system $\Sigma$  subject to the constraint set $(X,U)$ if $S \subseteq X$ and 
        \begin{equation}
        \label{eqn:rci_charac}
        \forall ~ x~ \in S, \exists u \in U \mbox{ such that } A x + B u  \in  S 
    \end{equation} 
\end{definition}

Because controlled invariance is stable under unions, the existence of a maximal controlled invariant set is guaranteed. 

\begin{definition}
\label{def:contr_invar}
Consider the system $\Sigma$ in (\ref{dis_sys}) and let $X \subseteq \mathcal{X}$, $U \subseteq \mathcal{U}$ be the constraint sets on the states,  inputs, respectively. The set $K \subseteq \mathcal{X}$ is the \emph{maximal} controlled invariant for the system $\Sigma$ and constraint set $(X,U)$ if:
\begin{itemize}
    \item $K \subseteq \mathcal{X}$ is a controlled invariant for the system $\Sigma$ and constraint set $(X,U)$;
    \item $K$ contains any controlled invariant for the system $\Sigma$ and constraint set $(X,U)$.
\end{itemize}
\end{definition} 
\subsection{Backward Reachable Set}
For discrete-time systems, the backward reachable set comprises all states from which a target set can be reached under admissible control inputs \cite{liu2023convergencebackwardreachablesets, 1618830}. 

\begin{definition} Consider the system $\Sigma$ in \eqref{dis_sys} . Given a set $  H \subseteq \mathbb{R}^n$, the one-step backward reachable set $\operatorname{Pre}_{\Sigma}\left(H, X,U\right)$ of $H$ with respect to the system $\Sigma$ and constraint set $X,U$, is  define by:
$$
\Pre_{\Sigma}\left(H, X,U\right)=\left\{x \mid x \in X, \exists u \in U \, A x+B u\in H\right\} .
$$
\end{definition}
Backward  reachability  can  be used  to characterise controlled invariant set.
\begin{proposition}\cite{liu2023convergencebackwardreachablesets}
Consider the system $\Sigma$ in \eqref{dis_sys} with constraint set $(X,U)$.The following results are true: 
\begin{itemize}
    \item A subset $H\subseteq X$ is a controlled invariant  if and only if $H\subseteq \Pre_{\Sigma}\left(H, X,U\right)$.
    \item If a subset $H\subseteq X$ is a controlled invariant  then $\Pre_{\Sigma}\left(H, X,U\right)$ is a controlled invariant.  
    \item The set $X_{max}$,  the maximal controlled invariant set or the system $\Sigma$ and constraint set $(X,U)$ satisfy   $X_{max}= \Pre_{\Sigma}\left(X_{max}, X,U\right)$
\end{itemize} 
\end{proposition}
 The \emph{$k$-step backward reachable set} is defined recursively as follows
$ H^{0} = H, \, H^{k+1} = \operatorname{Pre}_{\Sigma}(H^{k}, X, U).
$
If the set $X$ is compact, the $k$-step backward algorithm is guaranteed to converge. 
This observation leads to two common formulations: 
the \emph{outside-in} algorithm, initialized with $H^{0} = X$, 
and the \emph{inside-out} algorithm, initialized with $H^{0} = H$ is a controlled invariant.
\section{Trajectory-based Characterizations of Controlled Invariants}
\label{sec:traj}

In this section, we introduce a new notion of convex feasibility. We then show that convex feasible trajectories can be used to construct controlled invariant sets.

\begin{definition}
\label{def:feas}
Consider the system $\Sigma$ in (\ref{dis_sys}) and let $X \subseteq \mathcal{X}$, $U \subseteq \mathcal{U}$ be the constraint sets on the states, inputs respectively. A point $x_0 \in X$ is said to be {\it open-loop convex feasible} with respect to the constraint set $(X,U)$ if there exists $\epsilon >0$ an input trajectory $\mathbf{u}:\mathbb{N}_{\geq 0} \rightarrow U$ and $N>n+1$ such that 
\begin{equation}
\label{eqn:feas1o}
    \Phi(t,x_0,\mathbf{u})  \in X, \quad \forall~ 0\leq t \leq N-1 
\end{equation}
and 
\begin{equation}
\label{eqn:feas2o}
 \Phi(N,x_0,\mathbf{u}) + \mathcal{B}_{\epsilon}(0) \subseteq \ch\left(\bigcup\limits_{t= 0}^{N-1}\{\Phi(t,x_0,\mathbf{u})\}\right).
\end{equation}

A point $x_0 \in X$ is said to be {\it closed-loop convex feasible} with respect to the constraint set $(X,U,D)$ if there exists $\epsilon >0$, a controller $\kappa: X \rightarrow U$ and $N>n+1$ such that 
\begin{equation}
\label{eqn:feas1c}
    \Phi_{\kappa}(t,x_0)  \in X, \quad \forall~ 0\leq t \leq N-1 
\end{equation}
and 
\begin{equation}
\label{eqn:feas2c}
\Phi_{\kappa}(N,x_0) + \mathcal{B}_{\epsilon}(0) \subseteq \ch\left(\bigcup\limits_{t= 0}^{N-1} \{\Phi_{\kappa}(t,x_0)\}\right).
\end{equation}
\end{definition}

\begin{figure}[ht]
    \centering
\tikzset{every picture/.style={line width=0.75pt}} 
\resizebox{0.5\linewidth}{!}{\begin{tikzpicture}[x=0.75pt,y=0.75pt,yscale=-1,xscale=1]

\draw  (50,356.9) -- (518,356.9)(96.8,23) -- (96.8,394) (511,351.9) -- (518,356.9) -- (511,361.9) (91.8,30) -- (96.8,23) -- (101.8,30)  ;
\draw  [color={rgb, 255:red, 255; green, 255; blue, 255 }  ,draw opacity=1 ][fill={rgb, 255:red, 208; green, 2; blue, 27 }  ,fill opacity=1 ] (158.41,293.42) .. controls (158.41,291.45) and (156.3,289.85) .. (153.7,289.85) .. controls (151.11,289.85) and (149,291.45) .. (149,293.42) .. controls (149,295.4) and (151.11,297) .. (153.7,297) .. controls (156.3,297) and (158.41,295.4) .. (158.41,293.42) -- cycle ;
\draw  [color={rgb, 255:red, 255; green, 255; blue, 255 }  ,draw opacity=1 ][fill={rgb, 255:red, 208; green, 2; blue, 27 }  ,fill opacity=1 ] (156.84,70.58) .. controls (156.84,68.6) and (158.95,67) .. (161.54,67) .. controls (164.14,67) and (166.25,68.6) .. (166.25,70.58) .. controls (166.25,72.55) and (164.14,74.15) .. (161.54,74.15) .. controls (158.95,74.15) and (156.84,72.55) .. (156.84,70.58) -- cycle ;
\draw  [color={rgb, 255:red, 255; green, 255; blue, 255 }  ,draw opacity=1 ][fill={rgb, 255:red, 208; green, 2; blue, 27 }  ,fill opacity=1 ] (457.91,94.41) .. controls (457.91,92.43) and (460.02,90.83) .. (462.62,90.83) .. controls (465.21,90.83) and (467.32,92.43) .. (467.32,94.41) .. controls (467.32,96.38) and (465.21,97.98) .. (462.62,97.98) .. controls (460.02,97.98) and (457.91,96.38) .. (457.91,94.41) -- cycle ;
\draw  [color={rgb, 255:red, 255; green, 255; blue, 255 }  ,draw opacity=1 ][fill={rgb, 255:red, 208; green, 2; blue, 27 }  ,fill opacity=1 ] (473.59,224.31) .. controls (473.59,222.33) and (475.7,220.73) .. (478.3,220.73) .. controls (480.89,220.73) and (483,222.33) .. (483,224.31) .. controls (483,226.28) and (480.89,227.88) .. (478.3,227.88) .. controls (475.7,227.88) and (473.59,226.28) .. (473.59,224.31) -- cycle ;
\draw  [color={rgb, 255:red, 255; green, 255; blue, 255 }  ,draw opacity=1 ][fill={rgb, 255:red, 144; green, 19; blue, 254 }  ,fill opacity=1 ] (288.56,182.6) .. controls (288.56,180.62) and (290.66,179.02) .. (293.26,179.02) .. controls (295.86,179.02) and (297.97,180.62) .. (297.97,182.6) .. controls (297.97,184.57) and (295.86,186.17) .. (293.26,186.17) .. controls (290.66,186.17) and (288.56,184.57) .. (288.56,182.6) -- cycle ;
\draw    (161.54,74.15) -- (153.7,289.85) ;
\draw    (478.3,224.31) -- (158.41,293.42) ;
\draw    (462.62,97.98) -- (478.3,224.31) ;
\draw    (166.25,70.58) -- (457.91,94.41) ;

\draw (129,22) node [anchor=north west][inner sep=0.75pt]  [font=\LARGE] [align=left] {$\displaystyle x_{0}$};
\draw (464,55) node [anchor=north west][inner sep=0.75pt]  [font=\LARGE] [align=left] {$\displaystyle x_{1}$};
\draw (153.7,297) node [anchor=north west][inner sep=0.75pt]  [font=\LARGE] [align=left] {$\displaystyle x_{2}$};
\draw (485,229) node [anchor=north west][inner sep=0.75pt]  [font=\LARGE] [align=left] {$\displaystyle x_{3}$};
\draw (302.92,190.1) node [anchor=north west][inner sep=0.75pt]  [font=\LARGE] [align=left] {$\displaystyle x_{4}$};

\end{tikzpicture}
}
\caption{Convex feasible Trajectory initiated at $x_0$. The points $x_1=\Phi(1,x_0, \mathbf{u})$ at the top right, $x_2 = \Phi(2,x_0, \mathbf{u})$  and  $x_3=\Phi(3,x_0, \mathbf{u})$ at the bottom left and right respectively.  Finally $ x_4 = \Phi(  4,x_0, \mathbf{u}) \in \ch\left(\bigcup\limits_{t=0}^{3}\{ \Phi(t,x_0, \mathbf{u})\} \right)$.}
    \label{fig:feas_traj}
    \vspace{-15pt}
\end{figure}
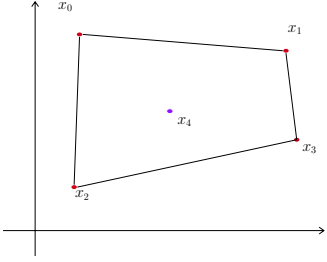
\begin{remark}
\label{rem:strict}
A more general formulation allowing $\epsilon \geq 0$ and $N \geq 0$ can be considered. 
Most of the properties established in this paper extend to this setting, with the exception of Theorem~4. 
In particular, the condition $\epsilon > 0$ is required in Theorem~4 to ensure convergence towards the maximal controlled invariant, 
whereas $\epsilon = 0$ is sufficient for Theorem~2.
\end{remark}

Intuitively, a point is \emph{convex feasible} if the state reached at time $N$ lies in the interior of the convex hull of the trajectory over the horizon $[0, N-1]$, i.e. $x_N \in \Int\!\big(\ch\{x_0, \dots, x_{N-1}\}\big).$
An illustration is provided in Fig.~\ref{fig:feas_traj}. 
The following result establishes a connection between convex feasibility and controlled invariance.

\begin{theorem}
     
\label{prop:feas_inv}
Consider the system $\Sigma$ in (\ref{dis_sys}) and let $X \subseteq \mathcal{X}$ and $U \subseteq \mathcal{U}$  be the constraint sets on the states, inputs, and disturbances, respectively, where the set $X$ is closed convex,  $U$ is compact convex.
\begin{itemize}
    \item[(i)] If $x_0 \in X$ is open-loop convex feasible w.r.t the constraint set $(X,U)$, then there exists an input trajectory $\mathbf{u}:\mathbb{N}_{\geq 0} \rightarrow U$ and $N \in \mathbb{N}_{>0}$ such that the set  \begin{equation} S= \cl\left(\ch\left(\{\Phi(t,x_0,\mathbf{u})\mid 0\leq t \leq N-1\}\right)\right)    \end{equation} is a controlled invariant for the system $\Sigma$ and constraint set $(X,U)$;

    \item[(ii)] If $x_0$ is closed loop convex feasible w.r.t the constraint set $(X,U)$, then there exists a controller $\kappa:X \rightarrow U$ and $N \in \mathbb{N}_{>0}$ such that the set  \begin{equation} S=\cl\left(\ch\left(\{\Phi_{\kappa}(t,x_0)\mid 0\leq t \leq N-1\} \right)\right)   \end{equation} is a controlled invariant for the system $\Sigma$ and constraint set $(X,U)$.
\end{itemize}

\end{theorem}
\textbf{\underline{Proof' sketch:}}
Any $x \in S$ is a convex combination of trajectory states. By linearity, its successor $x^+$ is the same combination of their successors. The terminal feasibility condition ensures these successors remain in $S$, implying controlled invariance.

At first glance, open- and closed-loop convex feasibility may seem unrelated; however, the following result demonstrates that they are equivalent.  

\begin{proposition}
\label{prop:feas_equiv}
    Consider the system $\Sigma$ in (\ref{dis_sys}) and let $X \subseteq \mathcal{X}$, $U \subseteq \mathcal{U}$  be the constraint sets on the states and inputs, respectively, where the set $X$ is closed convex. $x_0$ is open-loop convex feasible for constraint set  $(X,U)$ if and only if  $x_0$ is closed loop convex feasible for constraint set $(X,U)$.
\end{proposition}

The next result shows how invariants from convex feasible trajectories approximate the maximal controlled invariant set.

\begin{theorem}
\label{thm:stri_feas}
    Consider the system $\Sigma$ in (\ref{dis_sys}) and let $X \subseteq \mathcal{X}$, $U \subseteq \mathcal{U}$ be the constraint sets on the states and  inputs respectively. We suppose that $X,U$ are compact convex sets. If  $x_0 \in X$ is  {\it strictly open-loop convex feasible} with respect to the constraint set $(X,U)$, 
    then: 
    \begin{itemize}
        \item the set $H = \ch\,(\Phi(0, x_0,\mathbf{u}), \dots, \Phi(N-1, x_0, \mathbf{u}))$ is a  controlled invariant.  
    \item The sequence defined by $H_0 = H$, $H_{t+1} = \Pre_{\Sigma}\left(H_{t}, X,U, D\right)$ is composed of controlled invariant sets and converges in the Hausdorff metric to the maximal controlled invariant set. 
    \end{itemize}
\end{theorem}
\textbf{\underline{Proof' sketch:}}
If the terminal point lies in the interior of $H$, control laws can preserve this property, making strict feasibility recursive. Moreover, $H$ satisfies $S \subseteq H \subseteq S_N$. Since $S_N$ expands toward the maximal controlled invariant set $X_{\max}$ as $N$ grows, $H$ also converges to $X_{\max}$.
\begin{remark}
These results extend, at least partially, to perturbed and Linear parameter varying systems.
\end{remark}
\section{Application to MPC}
\label{sec:MPC}
We address recursive feasibility in MPC for constrained linear systems without precomputed terminal invariant sets. While classical schemes rely on computationally expensive terminal invariant constraints, we propose an alternative formulation based on the convex feasible trajectories introduced in Section~\ref{sec:traj}.

The standard finite-horizon MPC problem is given by
\begin{equation}
\label{eqn:MPC}
\begin{aligned}
\min_{\mathbf{u}_t^N} \quad & \sum_{k=0}^{N-1} \ell(x_{k|t}, u_{k|t}) + L(x_{N|t}) \\
\text{s.t.} \quad 
& x_{k+1|t} = A x_{k|t} + B u_{k|t}, \\
& x_{k|t} \in X,\; u_{k|t} \in U,\quad k=0,\dots,N-1, \\
& x_{N|t} \in X_f,\quad x_{0|t} = x_t
\end{aligned}
\end{equation}
where $\ell$ and $L$ are stage and terminal costs. At iteration $t$, the problem yields the optimal control sequence $\mathbf{u}_t^N$ given state $x_t$. The terminal set $X_f$ is precomputed offline.
\begin{definition} 
    Given a state $x_t \in X$ and a terminal constraint set $X_f \subseteq X$, a finite control sequence $\mathbf{u}_t^N =\{u_{0\mid t}, \dots, u_{N-1\mid t}\}$, with its associated trajectory $\mathbf{x}_t^N = \{x_{0\mid t} , \dots x_{N-1\mid t}\}$ is said to be feasible if 
    $$\begin{array}{rl}
         & x_{k \mid t} \in X, u_{k \mid t} \in U, k=0, \ldots, N-1 ; \\
& x_{N \mid t} \in X_f .
    \end{array}
    $$
    We denote by $\mathbf{U}\left(x_t, X_f\right)$ the set of all feasible control sequences $\boldsymbol{u}_t^N$ for state $x_t \in \mathbb{X}$ and terminal constraint set $X_t \subseteq X$.
\end{definition}

In particular, when the sets $U,X,X_f$ are compacts, and $\mathbf{U}\left(x_t, X_f\right) \neq \emptyset$, the optimization problem in (\ref{eqn:MPC}) admits at least one minimizer. 

An MPC scheme is recursively feasible if feasibility at $t$ ensures it for $t+1$ under the applied control. While a terminal invariant set $X_f$ guarantees this property \cite{mayne2000constrained}, its computation is often intractable. We replace this a priori constraint with the trajectory-based convex feasibility condition. Assuming convex $X$ and $U$, we consider the following problem:
\begin{equation}
\label{eqn:MPC_feas}
\begin{aligned}
\min_{\mathbf{u}_t^N} \quad & \sum_{k=0}^{N-1} \ell(x_{k|t}, u_{k|t}) + L(x_{N|t}) \\
\text{s.t.} \quad 
& x_{k+1|t} = A x_{k|t} + B u_{k|t}, \\
& x_{k|t} \in X,\; u_{k|t} \in U,\quad k=0,\dots,N-1, \\
& x_{N|t} \in \Int\big(\ch\{x_{0|t}, \dots, x_{N-1|t}\}\big), \\
& x_{0|t} = x_t.
\end{aligned}
\end{equation}
The terminal constraint enforces that the predicted terminal state lies in the interior of the convex hull of the preceding trajectory, which induces invariance without requiring an explicit terminal set.
\begin{proposition}
\label{pro:mpc}
The MPC scheme defined in \eqref{eqn:MPC_feas} is recursively feasible.
\end{proposition}
The above result requires an initial feasible solution. Once available, feasibility can be propagated using convex feasibility, allowing a successor solution to be constructed within the convex hull of previously computed trajectories.

Thus, recursive feasibility is ensured without explicitly computing a terminal invariant set. However, this introduces a non-convex constraint due to bilinear dependence in the convex hull condition, making the optimization problem more challenging and potentially suboptimal.

A practical remedy is to first compute a feasible trajectory with the proposed method, then extract an invariant set approximation to use as a terminal constraint in a standard MPC scheme. This hybrid approach combines reduced offline complexity with the convergence properties of classical MPC.
\section{Computing invariant sets}
\label{sec:computation}

We show how the convex feasibility conditions can be formulated as a finite-dimensional optimization problem to compute controlled invariant sets. 
We focus on the nominal case and assume that the constraint sets are convex polytopes. 
 \begin{assumption}
 \label{ass:2} The sets $X,U$  are closed convex polyhedral sets defined for some matrices $H_x,q,H_u,q_u$ of appropriate dimensions as follows :
     \begin{itemize}
     \item $X=\ch(\{v_1, v_2,\dots, v_{n_x}\}) = \{x \in \mathbb{R}^n \mid H_x x\leq q\}$
     \item  $U = \ch(\{u_0, u_1,\dots, u_{n_u}\})= \{u \in \mathbb{R}^m \mid H_u u\leq q_u\}$
     \end{itemize}
 \end{assumption}
We search for a convex feasible trajectory $x_0,\dots,x_{N-1}$.
The convex feasibility conditions can be encoded as the following optimization problem:
\begin{align}
\label{eqn:opt}
\min \quad & -d \\
\text{s.t.} \quad 
& x_{i+1} = A x_i + B u_i, \\
& x_i \in X,\; u_i \in U, \quad i=0,\dots,N-1, \\
& x_N + d s_j = \sum_{i=0}^{N-1} \lambda_{ij} x_i, \\
&\label{eqn:opt_fin} \sum_{i=0}^{N-1} \lambda_{ij} = 1,\quad \lambda_{ij} \ge 0,\quad d > 0.
\end{align}
with $s_j = (-1)^r e_{k}$ with $1\leq k \leq n , \, r \in \{0,1\}$, $j = rn + k$, $e_k$ is the basis vector. The constraints enforce the convex feasibility condition, while $d$ measures the strictness of the interior condition.   

\begin{theorem}
The optimization problem \eqref{eqn:opt}-\eqref{eqn:opt_fin} is feasible if and only if there exists a convex feasible trajectory for system $\Sigma$ under constraints $(X,U)$.
\end{theorem}

For the closed-loop case, we restrict the inputs to a linear state-feedback law $u_i = K x_i + b$, where $K \in \mathbb{R}^{m \times n}$ and $b \in \mathbb{R}^m$, subject to $K x_i + b \in U$. The prediction horizon $N$ is treated as a design parameter, typically determined by increasing its value until the problem becomes feasible.

The resulting optimization program involves $N+1$ states, $N$ inputs, and auxiliary variables $\lambda_{i,j}$ for convex combinations. The formulation is governed by linear dynamics and constraints, with bilinear terms arising from the coupling of states and convex coefficients.

The scalar $d$ serves as a proxy for the invariant set size, ensuring the set contains an $\ell_1$-ball of radius at least $d$. While this keeps the constraint count low, it may lead to conservatism in high-dimensional spaces. This can be mitigated by employing box-type terminal constraints or volume surrogates~\cite{mejari_data-driven_2023}.
\section{Numerical Examples}
\label{sec:Examples}
We validate our framework through three numerical studies: i) an uncontrollable system demonstrating the method's broad applicability; ii) a high-dimensional truck-trailer system to evaluate scalability against state-of-the-art methods~\cite{anevlavs-lui-ozay_2024}; and iii) a recursively feasible MPC design. Source code is available at: \href{https://github.com/JUNIORWAFO/Convex_Feasibility_Repo}{GitHub}.
\subsection{Example 1}
Consider the model defined by:
\[
    \begin{pmatrix}
         x^+_1  \\
         x^+_2
    \end{pmatrix} = 
    \begin{pmatrix}
         0.5 & 1 \\
         0 & -0.5
    \end{pmatrix} 
    \begin{pmatrix}
         x_1  \\
         x_2
    \end{pmatrix} + 
    \begin{pmatrix}
         1 \\
         0
    \end{pmatrix} u,
\]
with  $u \in[-1,1]$. The safe set is the same as in Example 1 from  \cite{9781450370189}. 
The system is not controllable, which makes the computation of invariant sets more challenging. We compare three approaches: the MPT toolbox, the method of \cite{anevlavs-lui-ozay_2024}, and the proposed method in this paper. The method of \cite{anevlavs-lui-ozay_2024} returns an empty set, while the MPT toolbox successfully computes the maximal invariant set. 
The proposed approach generates a non-trivial invariant set in approximately 1 second. The horizon $N=7$ was selected empirically as the smallest value ensuring feasibility of the optimization problem and yielding a non-trivial invariant set. Increasing $N$ led to marginal volume improvements at the cost of higher computational complexity.
Applying the backward reachable algorithm to this set allows recovering the maximal invariant set, with a total computation time of 2.7 seconds. 
\begin{figure}[ht]
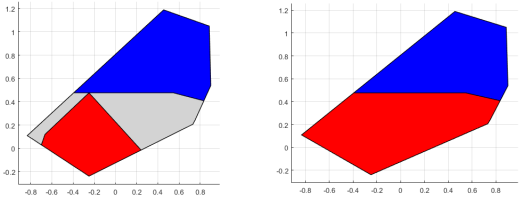

    \begin{minipage}{0.4\linewidth}
        \centering
        \includesvg[width=\linewidth]{images/compare_1.svg}
    \end{minipage}
    \begin{minipage}{0.45\linewidth}
        \centering
        \includesvg[width=\linewidth]{images/compare_2.svg}
    \end{minipage}
    \caption{Invariant set computation. 
(a) Comparison between the maximal invariant set (MPT) and the set obtained with the proposed method. 
(b) Convergence of the proposed set to the maximal invariant set via backward iteration.}
\vspace{-15pt}
    \label{fig:set_1}
\end{figure}

\subsection{Truck with M Trailers}
Consider the continuous-time system of a truck with $M$ trailers, where the state $x = [d_1, \dots, d_M, v_0, \dots, v_M]^\top$ consists of relative distances and velocities. The dynamics are governed by:
$$
\begin{aligned}
    \dot{d}_i &= v_{i-1} - v_i, \\
    \dot{v}_0 &= \frac{k_s}{m} d_1 - \frac{k_d}{m} v_0 + \frac{k_d}{m} v_1 + u, \\
    \dot{v}_i &= \frac{k_s}{m}(d_i - d_{i+1}) + \frac{k_d}{m}(v_{i-1} - 2v_i + v_{i+1}), \\
    \dot{v}_M &= \frac{k_s}{m} d_M - \frac{k_d}{m} v_M + \frac{k_d}{m} v_{M-1},
\end{aligned}
$$
for $i=1, \dots, M$. The system is discretized with sampling time $T_s$, resulting in a discrete-time linear system of dimension $n = 2M+1$. We used parameters  from \cite{9781450370189}.

We evaluate the proposed method against the MPT toolbox and the approach of~\cite{anevlavs-lui-ozay_2024} as the system dimension increases. As shown in Table~\ref{tab:comparaison}, the MPT toolbox fails to converge for $M \geq 3$ within the iteration limit (100), while~\cite{anevlavs-lui-ozay_2024} encounters numerical issues at $M \geq 3$ due to the complexity of volume computations.

In contrast to the implicit polyhedral approach of~\cite{anevlavs-lui-ozay_2024}, which relies on computationally intensive set projections, our method employs a trajectory-based formulation. By constructing invariant sets from convex combinations of feasible trajectories, we bypass explicit set propagation and ensure computational tractability via finite-dimensional optimization. While this trades geometric exactness for potential conservatism, our approach scales effectively to higher dimensions and can be refined using multiple trajectories or backward reachability.

\begin{table}[ht]
    \centering
    \resizebox{\linewidth}{!}{
    \begin{tabular}{lllll}
    \hline
         & $M=1$ &  $M=2$  & $M=3$ & $M=4$\\
         \hline
         \hline
    MPT (100 iterations)    & 0.05 & 0.66 & $>52$ & $>401$
    \\
    Volume & 21.6826 & 20.7586 & NA & NA \\
    \hline 
     Method \cite{anevlavs-lui-ozay_2024}    & 0.8 & 1.9 & 9.6 & 307.2
    \\
    Volume & 21.6826 & 18.9066 & NA & NA \\
    \hline 
     Our approach (N=12) & 1.8 & 0.8 & 2.3 & 3.14
    \\
    Volume & 5.9217 & 0.131 & $3.2\times  10 ^{-6}$ & $7\times  10 ^{-13}$ \\ 
    \hline 
     Our approach (N=12)  & 2.6 &8.76 & NA & NA \\
     Backward Fixed Point & & & &
    \\
    Volume &21.6826 & 20.7586 & NA & NA \\ 
    \hline 
     Our approach   & NC & NC & 20.6 & 151 \\
    ($12\leq N \leq 24$) & & & & \\
    Volume &NC & NC  & 0.025 & $6.52\times 10^{-5}$
    \end{tabular}}
    \caption{Comparison of computation time (s) and invariant set volume for the different methods. NA indicates cases where the computation did not return a result. NC indicate cases  where the computation was not attempted}
    \vspace{-15pt}
    \label{tab:comparaison}
\end{table}
\subsection{The Coupled Tanks}
In this section, we consider the coupled tanks system from \cite{coupled_tanks_2}, described by 
$$x^+  = A x + B u. $$
The constraints on the states and inputs are given by:
\[
X = \{x \in \mathbb{R}^4 \mid \underline{x} \leq x \leq \overline{x}\}, \quad U = \{u \in \mathbb{R}^2 \mid \underline{u} \leq u \leq \overline{u}\},
\]
\[
\begin{aligned}
& \underline{u}=-[4.53 ;5.56]\times 10^{-4}, \quad \overline{u}=[4.53 ;5.56]\times 10^{-4}, \\
& \underline{x}=[-0.45 ;-0.46 ;-0.45 ;-0.46],  \overline{x}=[0.71 ; 0.7 ; 0.65 ; 0.64].
\end{aligned}
\]
The system matrices are taken from \cite{BRIAO20219741}.
From an initial condition of $x_0 = [0.5;0.5;0.5;0.5]$, we consider the finite horizon MPC problem in (\ref{eqn:MPC_feas}) with a  cost function $$ J_N\left(x_t, \mathbf{u}_t^N\right) = \sum_{k=0}^{N} (x_{k\mid t}-x_r)^T Q (x_{k\mid t}-x_r)  + \sum_{k=0}^{N-1} u_k^T R u_k $$
where $x_r$ represents the reference trajectory and the matrices $Q$ and $R$ are given by:
$$\scriptstyle{Q=\left[\begin{array}{cccc}
1.5 & 0 & 0 & 0 \\
0 & 3.8 & 0 & 0 \\
0 & 0 & 10.1 & 0 \\
0 & 0 & 0 & 27.3
\end{array}\right], \quad R=\left[\begin{array}{cc}
1 & 0 \\
0 & 1 
\end{array}\right]}.$$

The results confirm that the scheme ensures recursive feasibility while satisfying state and input constraints. 
\begin{figure}[ht]
\centering

\begin{subfigure}{\linewidth}
    \centering
    \includesvg[width=0.8\linewidth]{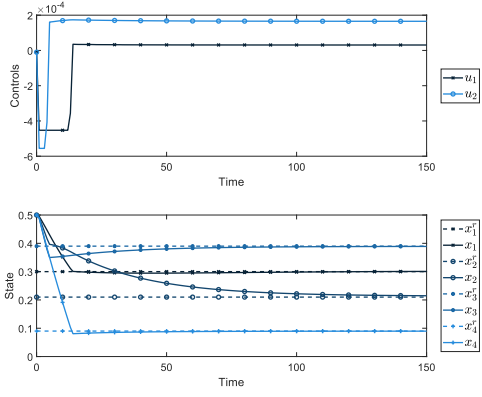}
    \caption{Evolution of the input (top) and state (bottom) of the resulting closed-loop trajectories using the proposed MPC scheme for a reference point $x_r = [0.3,0.21,0.39,0.1]$ and a horizon $N=40$.}
    \label{fig:MPC_other}
\end{subfigure}
\hfill
\begin{subfigure}{\linewidth}
    \centering
    \includesvg[width=0.8\linewidth]{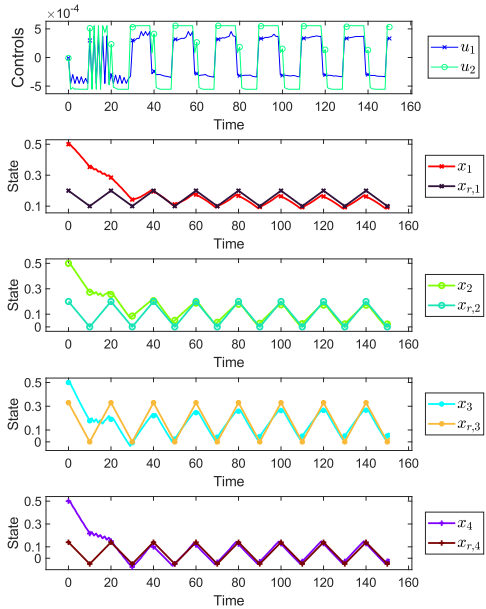}
    \caption{Evolution of the input (top) and states of the resulting closed-loop trajectories using the proposed MPC scheme for a given reference signal $x_r$ and for a horizon $N=40$.}
    \label{fig:MPC_traj}
\end{subfigure}

\caption{Closed-loop trajectories under the proposed MPC scheme for $N=40$.}
\vspace{-15pt}
\label{fig:MPC_combined}
\end{figure}

\section{Conclusion}
In this paper, we proposed a trajectory-based framework for the synthesis of controlled invariant sets for linear systems, providing an alternative to classical set-theoretic approaches. By leveraging convex feasibility of trajectories, we derived practical conditions for invariant set construction and developed a recursively feasible MPC scheme without requiring a precomputed terminal invariant set. The effectiveness of the approach was demonstrated on representative examples, highlighting its computational advantages in settings where traditional methods become intractable. 

Future work will focus on improving the computational efficiency of the proposed optimization problems and extending the framework to handle richer specifications, including convex temporal logic constraints \cite{saoud2024temporallogicresiliencedynamical}.
\section*{Acknowledgment}
AI tools (ChatGPT and Gemini) were used for proofreading only; all technical results were derived by the authors.

\ifitsdraft
\section{Auxilliary results}

\subsection{Convex Sets}
In this section we provide some definitions and characterizations of convex sets.
\begin{definition} A subset C of $\mathbb{R}^n$ is a convex set if for all $x,y \in C$, $[x,y] \subseteq C$.
\end{definition}
\begin{definition}\cite{hiriart-urruty_fundamentals_2012}
Let $S$ be a subset of  $\mathbb{R}^n$. The convex hull of S can be defined as: 
     $$\ch(S) = \{ \sum\limits_{i=1}^k \alpha_i x_i, 0 \leq \alpha_i, \sum\limits_{i=1}^k \alpha_i = 1, x_i \in S, k \in \mathbb{N}\}$$.
The closed convex hull of S is the set $\cl(\ch(S))$.\footnote{In general $\cl(\ch(S)) \neq \ch(\cl(S))$, see \cite[Chapter A, fig 1.4.1]{hiriart-urruty_fundamentals_2012}.}
\end{definition}
We also provide the definition of extreme points of a closed convex set.
\begin{definition}
    Consider a closed convex set $C \subseteq \mathbb{R}^n$. A point $p \in C$ is called an extreme point of $C$ if it does not lie between any two distinct points of $C$. That is, if there does not exist $x, y \in C$ such that $x \neq y$ and $p \in ]x,y[$. The set of all extreme points of $C$ is denoted by extreme $\ex(C)$.
\end{definition}
The following definition introduces the notions of affine closure, relative interior, and boundary.
\begin{definition}
    Let $C$ be a subset of $\mathbb{R}^n$. 
    \begin{itemize}
        \item  The Affine closure of $C$, denoted $\af(C)$is defined as follow: $$\af(C)  = \{ \sum\limits_{i=1}^{n+1} \alpha_i x_i, \sum\limits_{i=1}^{n+1} \alpha_i = 1, x_i \in C\}.$$ 
        \item  The  relative interior of a set C is defined by $$\ri C = \{x \in C \mbox{ such that } \exists \delta>0 \mid \mathcal{B}_{\delta}(x)\cap \af(C) \subseteq C\}$$ and the relative exterior $\rdelta C = \cl(C)\setminus \ri C$
    \end{itemize}
\end{definition}
\begin{corollary}
\label{cor:1}
     Consider any points $x_0, x_1, \dots, x_n \in \mathbb{R}^d$. Let $x \in \mathbb{R}^d$. If $x \in \ri(\ch(x_0, \dots x_n))$, then: $$\af(\{x_0, \dots, x_n\}) = \af(\{x, x_1, \dots, x_n\})$$
\end{corollary}
\label{sec:aux}
\subsection{Controlled Invariance}
In this section, we introduce a selection of auxiliary results.
\begin{lemma}
    \label{prop:rci_struc}
    Consider the system $\Sigma$ in (\ref{dis_sys}) and let $X \subseteq \mathcal{X}$ and $U \subseteq \mathcal{U}$ be the constraint sets on the states and inputs, respectively. We suppose that $X$ is a closed convex set, and $U$ is a compact convex set. If $S$ is a  controlled invariant for system $\Sigma$ and constraint set  $(X,U)$ then the  sets $\cl(S)$, $\ch(S)$ and $\cl(\ch(S))$ are also controlled invariants.
\end{lemma}
Lemma \ref{prop:rci_struc} shows that  controlled invariance is closed with respect to convex hull and closure

\begin{lemma}
\label{thm1}
Consider the system $\Sigma$ in (\ref{dis_sys}) and let $X \subseteq \mathcal{X}$ and $U \subseteq \mathcal{U}$ be the constraint sets on the states and inputs, respectively. We suppose that X and U  are compact convex sets. Then the followings holds: 
A closed convex set $S \subseteq X$ is a  controlled invariant set for system $\Sigma$ and constraint set  $(X,U)$ if and only if:
\begin{multline}
     \forall x \in \ex(S), \, \exists u \in U\mbox{ such that }  \mbox{ } Ax + B u \in S .
     \label{eqn:prop_charact}
 \end{multline}
\end{lemma}
 
\fi
\ifitsdraft
\section{ proofs}
\label{sec:proofs}

\textcolor{blue}{\textbf{\underline{Proof of Lemma \ref{lem:Int_char}}}}
Let $H= \ch(x_0,\dots, x_n)$.

\textbf{\underline{Necessary condition}} Suppose that $x \in \ri( H)$.  By the definition of $H$ there exists $\alpha_0, \dots, \alpha_n $ such that $\begin{pmatrix}
        x\\1
    \end{pmatrix} = \sum\limits_{i=0}^{n} \begin{pmatrix}
        \alpha_i x_i \\
        \alpha_i
    \end{pmatrix}$ and  $\alpha_i \geq 0$.
    Let  $I  = \{i \mid \alpha_i = 0\}$.
    
    If $I$ is empty, we have the property. 
    
    Suppose $I$ is not empty. Consider any $i$ in $I$, either $x = x_i$ or $x\neq x_i$. Consider the case  $x \neq x_i$. Since $x \in \ri(H)$, there exists $\epsilon > 0$ such that
$\mathcal{B}_\epsilon(x) \cap \af(H) \subset H.
$ 
Let
$L := \{x_i + t(x - x_i) : t \in \mathbb{R}\} \subset \af(H)$ 
and define $ T:= \{ t \in \mathbb{R} : x_i + t(x - x_i) \in H \}.$
Since $H$ is convex, $T$ is an interval. Moreover, since $x \in \ri(H)$, there exists an open neighborhood of  $ 1 \in T$ such that  $
x_i + t(x - x_i) \in H \quad \text{for all } t \text{ close to } 1,
$
hence $T$ contains a nonempty open interval.

Therefore, $T$ is an interval with 2 endpoints since $H$ is compact. When these endpoints are finite, they correspond to two points in $L \cap \rdelta H$, both distinct from $x$. Furthermore, since

$\mathcal{B}_\epsilon(x) \cap \af(H) \subset H$
these boundary points cannot lie in $\mathcal{B}_\epsilon(x)$, and thus each of them is at distance at least $\epsilon$ from $x$.
    One of these corresponds to $t_{i}>1$, let's denote it $z_{i} \in \rdelta( H) $. We then have $$x = \frac{1}{t_{i}} z_{i} +  (1 - \frac{1}{t_{i}}) x_i = \gamma_i z_i + (1-\gamma_i) x_i$$ with $0<\gamma_i<1$. The case $x =x_i$ correspond to $\gamma_i = 0$ So we can write $0\leq \gamma_i<1 $. Since  $z_{i} \in H$, we have the existence of  $\alpha_{i,0}, \dots, \alpha_{i,n} $ such that $\begin{pmatrix}
        z_i\\1
    \end{pmatrix} = \sum\limits_{j=0}^{n} \begin{pmatrix}
        \alpha_{i,j} x_j \\
        \alpha_{i,j}
    \end{pmatrix}$ and  $\alpha_{i,j} \geq 0$.  Let $m = \#I$.  We can write: 
    \begin{align*}
                x = & \frac{1}{m+1}\left( x + \sum_{i\in I} x\right) \\
                    = & \frac{1}{m+1}\left( x + \sum_{i\in I} \left[\gamma_i z_{i} +  (1 - \gamma_i) x_i \right]\right) \\ 
                    = & \frac{1}{m+1}\left( \sum\limits_{j=0}^{n}
        \alpha_{j} x_j + \sum_{i\in I} \gamma_i\sum\limits_{j=0}^{n}
        \alpha_{i,j} x_j +   \sum_{i\in I} (1 - \gamma_i) x_i \right) \\
        = & \frac{1}{m+1}\left( \sum\limits_{j=0}^{n}
        \alpha_{j} x_j + \sum\limits_{j=0}^{n} \sum_{i\in I} \gamma_i
        \alpha_{i,j} x_j +   \sum_{i\in I} (1 - \gamma_i) x_i \right) \\ 
        = & \frac{1}{m+1}\left( \sum\limits_{j=0}^{n}\left[
        \alpha_{j}+ \sum_{i\in I} \gamma_i
        \alpha_{i,j} \right] x_j +   \sum_{i\in I} (1 - \gamma_i) x_i \right) \\
        = & \frac{1}{m+1}\left( \sum\limits_{i=0}^{n} \alpha^{1}_i x_i +   \sum_{i\in I} (1 - \gamma_i) x_i \right) \\
        = & \sum\limits_{i=0}^{n} \lambda_i x_i
    \end{align*}
    with $\alpha^{1}_i = \alpha_{i}+ \sum_{j\in I} \gamma_j
        \alpha_{j,i}$, $\lambda_i = \frac{\alpha^{1}_i}{m+1}$ if $i \notin I$ and   $\lambda_i = \frac{\alpha^{1}_i + \left(1-\gamma_i\right)}{m+1}$ if $i\in I$.  
    From the definition of $\alpha_i$, $\alpha_{i,j}$, $\gamma_i$ and $I$, one can get that $\lambda_i > 0$ for all $0 \leq i \leq n$.  We have 
    \begin{align*}
        \sum\limits_{i=0}^{n} \lambda_i= & \sum\limits_{i\in I} \lambda_i + \sum\limits_{i\notin I} \lambda_i \\
        = &  \sum\limits_{i\in I} \frac{\alpha^{1}_i + \left(1-\gamma_i\right)}{m+1} + \sum\limits_{i\notin I} \frac{\alpha^{1}_i}{m+1} \\
        = & \frac{\sum\limits_{i\in I} [\alpha^{1}_i + \left(1-\gamma_i\right)] + \sum\limits_{i\notin I} \alpha^{1}_i}{m+1} \\
        =& \frac{\sum\limits_{i\in I} \left(1-\gamma_i\right) + \sum\limits_{i=1}^n  \alpha^{1}_i}{m+1} \\ 
        = & \frac{\sum\limits_{i\in I} \left(1-\gamma_i\right) + \sum\limits_{i=0}^n  [ \alpha_{i}+ \sum\limits_{j\in I} \gamma_j
        \alpha_{j,i}}{m+1} ]\\
        = & \frac{\sum\limits_{i\in I} \left(1-\gamma_i\right) + \sum\limits_{i=0}^n   \alpha_{i}+ \sum\limits_{i=0}^n\sum\limits_{j\in I} \gamma_j
        \alpha_{j,i}}{m+1} \\
        = & \frac{ 1+  \sum\limits_{i\in I} \left(1-\gamma_i\right) + \sum\limits_{j\in I} \left(\sum\limits_{i=0}^n 
        \alpha_{j,i}\right) \gamma_j}{m+1}\\
     = &  \frac{ 1+  \sum\limits_{i\in I} \left(1-\gamma_i\right) + \sum\limits_{j\in I} \gamma_j}{m+1} =   \frac{ 1+  \sum\limits_{i\in I} [\left(1-\gamma_i\right) + \gamma_i]}{m+1} \\
     =& \frac{ 1+  \#I}{m+1} = 1 
    \end{align*}
    We get the desired result. 

 \textbf{\underline{Sufficient condition}} Consider any $x \in H$.  Suppose that there exists $\lambda_0,\dots,\lambda_n$ such that: $\left( \begin{array}{c}
            x \\
            1 
          \end{array} \right)= \sum\limits_{i=0}^n \lambda_i \left( \begin{array}{c}
            x_i \\
            1 
          \end{array} \right) , \lambda_i>0$ and $x \in \rdelta(H)$. Then there exists a $z \in H\setminus\{x\}$ such that the line $L_z: t \mapsto tx +(1-t)z$ intersects the relative boundary in only one other point other than $x$. In other word for all $t>1$, $ tx +(1-t)z \notin H$. We have: 
          \begin{equation*}
              tx +(1-t)z =  \sum\limits_{i=0}^n t\lambda_i x_i +  \sum\limits_{i=0}^n (1-t) \alpha_i x_i=  \sum\limits_{i=0}^n \lambda_{t,i} x_i
          \end{equation*}
          with $\lambda_{t,i} = t\lambda_i + (1-t) \alpha_i$, $\sum\limits_{i=0}^n \alpha_i = 1, \,  \alpha_i \geq 0$. We have 
          $$\sum\limits_{i=0}^n \lambda_{t,i} = t \sum\limits_{i=0}^n \lambda_{i}  +(1-t) \sum\limits_{i=0}^n \alpha_{i}= 1$$
          if for all $i$ , $0\leq \lambda_{t,i} \leq 1$ then $ tx +(1-t)z \in H$. 
          
          For all $i$, we have :
          \begin{align*}
              0\leq \lambda_{t,i} \leq 1 \Longleftrightarrow & 0\leq t\lambda_{i} + (1-t) \alpha_{i} \leq 1 \\
              \Longleftrightarrow &  -\alpha_i \leq t (\lambda_i - \alpha_i) \leq 1 - \alpha_{i}.
          \end{align*} 
        We have 3 cases: 
        \begin{itemize}
            \item $(\lambda_i - \alpha_i) > 0$ which yields,  
            $-\frac{\alpha_i}{\lambda_i - \alpha_i} \leq t \leq \frac{1-\alpha_i}{\lambda_i - \alpha_i}$
            \item $(\lambda_i - \alpha_i) < 0$ which yields: 
            $\frac{1-\alpha_i}{\lambda_i - \alpha_i} \leq t \leq \frac{-\alpha_i}{\lambda_i - \alpha_i}$
            \item $(\lambda_i - \alpha_i) =0$ which yields $t\in \mathbb{R}$.
        \end{itemize}
        Let $I_1 =  \{i\mid (\lambda_i - \alpha_i) > 0\}$ ,  $I_2 = \{i \mid (\lambda_i - \alpha_i) < 0\}$ and $I_3 = \{i \mid (\lambda_i - \alpha_i) = 0\} $. Using  the fact that  $z \neq x$ then  $\#I_1 + \# I_3>1$, and  that $\frac{1-\alpha_i}{\lambda_i - \alpha_i} > 1$ when $i \in  I_1$ , $\frac{-\alpha_i}{\lambda_i - \alpha_i} > 1$ when $i \in  I_2$, we  have  that :
        $$t_0 = \min(\min_{i \in I_1}(\frac{1-\alpha_i}{\lambda_i - \alpha_i}),\min_{i \in I_2}(\frac{-\alpha_i}{\lambda_i - \alpha_i}))>1$$
          
          Taking $t_1 = \frac{t_0+1}{2}> 1$, we get $z = t_1 x + (1-t_1) z \in H $. Which is in contradiction for all $t>1, t x + (1-t) z \notin H$.  
   In conclusion $x \in \ri(H)$
\medskip

\textcolor{blue}{\textbf{\underline{Proof of Theorem \ref{prop:feas_inv}:}}}
   
We will give proof for each item separately.
    \begin{enumerate}
    \item[(i)] Let $x \in \cl(\ch( \{\Phi(t,x_0,\mathbf{u})), 0\leq\, t \leq \,N-1\}) =S$. From  the definition of convex hull we have the existence of $\alpha_0, \dots, \alpha_{N-1}$,  such that for all $0 \leq i \leq N-1$, $0\leq \alpha_i$, $\sum\limits_{i=0}^{N-1} \alpha_i =1$,  and  $x = \sum\limits_{i=0}^{N-1} \alpha_i \Phi(i,x_0,\mathbf{u})$. For all $1 \leq i \leq N-1$, let us defines $x_i = : \Phi(i,x_0,\mathbf{u})$. Let $u = : \sum\limits_{i=0}^{N-1} \alpha_i \mathbf{u}(i)$. Let us defines $u_i =: \mathbf{u}(i)$ for all  $0\leq i \leq N$. We then have:
        \begin{equation}
            \begin{array}{rl}
                x^+ = & A x +B u   
                 =  \sum\limits_{i=0}^{N-1} \alpha_i \left[A x_i + B u_i \right] \\
                =&  \sum\limits_{i=0}^{N-1} \alpha_i \Phi(i+1,x_0,\mathbf{u})  
                \\
                = &\sum\limits_{i=0}^{N-2} \alpha_i x_{i+1} +  \alpha_{N-1}x \\ 
                = &\sum\limits_{i=0}^{N-2} \alpha_i x_{i+1} +  \alpha_{N-1}\sum\limits_{i=0}^{N- 1} \alpha_i x_{i} \\
                =& \sum\limits_{i=0}^{N- 1} \beta_i x_{i} \in S
            \end{array}
        \end{equation}
 where the last inclusion follows from (\ref{eqn:feas2o}) , the convexity of $S$ and  the fact that $\beta_0 = \alpha_{N-1}\alpha_0 $, $\beta_i = \alpha_{N-1}\alpha_i+ \alpha_{i-1}$. Hence, the set $S$ is then a  controlled invariant. 
         \item[(ii)] Let $x \in  \ch( \{\Phi_{\kappa}(t,x_0), 0\leq\, t \leq \,N-1\}) =S$. From the definition of convex hull, we have the existence of $\alpha_0, \dots, \alpha_{N-1}$ such that for all $0 \leq i \leq N-1 $, $0\leq \alpha_i$, $\sum\limits_{i=0}^{N-1} \alpha_i =1$ and  $x = \sum\limits_{i=0}^{N-1} \alpha_i \Phi_{\kappa}(i,x_0)$.  For all $1\leq i \leq N-1$ and $x_i = \Phi_{\kappa}(i,x_0)$. Let $u = \sum\limits_{i=1}^{n+1} \alpha_i \kappa(x_i)$. We have that:
        \begin{equation}
            \begin{array}{rl}
                x^+ = & A x +B u  
                 =  \sum\limits_{i=0}^{N-1} \alpha_i \left[Ax_i + B \kappa(x_i) \right] \\
                 =  & \sum\limits_{i=0}^{N-1} \alpha_i \Phi_{\kappa}(i+1,x_0) \\ 
                  = & \sum\limits_{i=0}^{N-2} \alpha_i x_{i+1} + \alpha_{N-1} x = \sum\limits_{i=0}^{N- 1} \beta_i x_{i} \in S
            \end{array}
        \end{equation}
 where the last inclusion follows from (\ref{eqn:feas2c}), the convexity of S and  that $\beta_0 = \alpha_{N-1} \alpha_0$ and  $\beta_i = \alpha_{i-1} + \alpha_{N} \alpha_i$. Hence, the set $S$ is then a controlled invariant.
     \end{enumerate}     
\medskip

\ifitsdraft

\textcolor{blue}{\textbf{\underline{Proof of Proposition \ref{prop:feas_equiv}:}}}

 \textbf{Sufficient Condition:} Suppose that $x_0$ is open-loop convex feasible for constraint set $(X,U)$. Then we have the existence of $\mathbf{u}:\mathbb{N}_{\geq 0} \rightarrow U$ and $N>0$ such that  (\ref{eqn:feas1o}) and (\ref{eqn:feas2o}) are satisfied. 
    We define the controller $\kappa:X \rightarrow U$ as follows: \begin{equation}  \label{eqn:controller_design}
        \kappa(x) = \begin{cases}
        \mathbf{u}(t) \mbox{ if } x = \Phi(t,x_0,\mathbf{u})\mid 0 \leq t \leq N-1\\ 
        \Single(U)
    \end{cases}
    \end{equation}
    Let us first show that $\Phi_{\kappa}(t, x_0) = \Phi(t, x_0,\mathbf{u})$ for $0\leq t \leq N$ by induction. Initially for $t=0$, $ x_0 = \Phi_{\kappa}(t, x_0) = \Phi(t, x_0,\mathbf{u})$. Consider any $0\leq t < N$ such that $\Phi_{\kappa}(t, x_0) = \Phi(t, x_0,\mathbf{u})$. We then  have: $$\Phi_{\kappa}(t+1, x_0) = A \Phi_{\kappa}(t, x_0)+ B \kappa(\Phi_{\kappa}(t, x_0))
    .$$ By the induction hypothesis $\Phi_{\kappa}(t, x_0) = \Phi(t, x_0,\mathbf{u})$. Using the definition of $\kappa$, we get that $\kappa(\Phi_{\kappa}(t, x_0)) = \mathbf{u}(t)$. Hence, one gets that $\Phi_{\kappa}(t+1, x_0) = \Phi(t+1, x_0,\mathbf{u})$. To conclude,  \begin{equation} \label{eqn:equi_clo_op_1}\Phi_{\kappa}(t, x_0) = \Phi(t, x_0,\mathbf{u})
    \end{equation} for all $0 \leq t \leq N$.

    Since $x_0$ is open-loop convex feasible, it follows that : 
    \begin{align*}
         \Phi_{\kappa}(t, x_0)& = \Phi(t, x_0,\mathbf{u}) \in X \mbox{ for } 0\leq t \leq N-1\\
       \mbox{ and }\Phi_{\kappa}(N, x_0,) & = \Phi(N, x_0,\mathbf{u})\\
        & \subseteq \ch\left(\{\Phi(t, x_0,\mathbf{u})\mid 0\leq t \leq N-1 \}\right) \\ 
        & \subseteq \ch\left(\{\Phi_{\kappa}(t, x_0)\mid 0\leq t \leq N-1 \}\right)
     \end{align*}
    where the first line is the consequence of condition (\ref{eqn:feas1c}), and the second and third inclusions come from condition (\ref{eqn:feas2c}) and (\ref{eqn:equi_clo_op_1}).  In conclusion $x_0$ is open-loop convex feasible for constraint set $(X,U)$.
   
    \textbf{Necessary Condition:} In that case starting from the controller $\kappa$ we can construct the input trajectory $\mathbf{u}:\mathbb{N} \rightarrow U$ as follows: for $0\leq t \leq N- 1$, $\mathbf{u}(t) = \kappa(\Phi_{\kappa}(t, x_0))$ and for $t \geq N$ ,  $\mathbf{u}(t) = \Single(U)$. The input trajectory $\mathbf{u}$ is well defined since $\Phi_{\kappa}(t, x_0)$ are singletons. We can show by induction that for all $0 \leq t \leq N$, \begin{equation}
    \label{eqn:equi_clo_op}
        \Phi_{\kappa}(t, x_0) = \Phi(t,x_0,\mathbf{u}).
        \end{equation}
    The proof of (\ref{eqn:equi_clo_op}) is similar to the one for the sufficient conditions and thus omitted. Since $x_0$ is closed-loop convex feasible, it follows that : 
    \begin{align*}
         \Phi(t, x_0,\mathbf{u})& = \Phi_{\kappa}(t, x_0) \in X \mbox{ for } 0\leq t \leq N-1\\
       \mbox{ and } \Phi(N, x_0,\mathbf{u}) & = \Phi_{\kappa}(N, x_0,  )\\
        & \subseteq \ch\left(\{\Phi_{\kappa}(t, x_0)\mid 0\leq t \leq N-1 \}\right) \\ 
        & \subseteq  \ch\left(\{\Phi(t, x_0,\mathbf{u})\mid 0\leq t \leq N-1 \}\right)
     \end{align*}
    where the first line is the consequence of condition (\ref{eqn:feas1c}), and the second and third inclusions come from condition (\ref{eqn:feas2c}) and (\ref{eqn:equi_clo_op}).  In conclusion $x_0$ is open-loop convex feasible for constraint set $(X,U)$.

\medskip
\fi 
\textcolor{blue}{\textbf{\underline{Proof of Theorem \ref{thm:stri_feas}}}}

We prove each statement separately. Since open-loop and closed-loop convex feasibility are equivalent, we restrict attention to the open-loop case.

\begin{itemize}

\item[1.] 
If $x_0 \in X$ is open-loop convex feasible, then it is also closed-loop convex feasible. By Theorem~\ref{prop:feas_inv}, the set $H$ is controlled invariant for $\Sigma$ under the constraints $(X,U)$.

\medskip

\item[2.] 
We show that convex feasibility is recursive and use this to construct a controlled invariant set $S$ such that
\[
S \subseteq H \subseteq S_N,
\]
where $S_k$ denotes the $k$-step backward reachable set of $S$.

\medskip

\noindent\emph{Step 1: One-step propagation.}
Let $x = \Phi(N,x_0,\mathbf{u})$. Since $x \in \Int(H) = \ri(H)$, by Lemma~\ref{lem:Int_char}, there exist $\alpha_i > 0$ such that
\[
x = \sum_{i=0}^{N-1} \alpha_i x_i, 
\qquad 
\sum_{i=0}^{N-1} \alpha_i = 1,
\quad x_i := \Phi(i,x_0,\mathbf{u}).
\]

Define $u := \sum_{i=0}^{N-1} \alpha_i u_i$, where $u_i := \mathbf{u}(i)$. Then
\[
\begin{aligned}
A x + B u
&= \sum_{i=0}^{N-1} \alpha_i (A x_i + B u_i) \\
&= \sum_{i=0}^{N-1} \alpha_i x_{i+1} = \sum_{i=0}^{N-2} \alpha_i x_{i+1} + \alpha_{N-1} x.
\end{aligned}
\]

Substituting the expression of $x$, we obtain
\[
A x + B u = \sum_{i=0}^{N-1} \alpha'_i x_i,
\]
where
\[
\alpha'_0 = \alpha_{N-1}\alpha_0,
\quad
\alpha'_i = \alpha_{i-1} + \alpha_{N-1}\alpha_i, \quad 1 \le i \le N-1.
\]

Since $\alpha_i > 0$, we have $\alpha'_i > 0$ and $\sum \alpha'_i = 1$, hence
\[
A x + B u \in \ri\big(\ch(x_1,\dots,x_N)\big) \subseteq \Int(H).
\]

\medskip

\noindent\emph{Step 2: Recursive construction.}
Define a control sequence $\mathbf{u}_1$ by
\[
\mathbf{u}_1(t) =
\begin{cases}
\mathbf{u}(t), & 0 \le t \le N-1, \\
\sum_{i=0}^{N-1} \alpha_i \mathbf{u}(i), & t = N.
\end{cases}
\]

Then
\[\begin{array}{ll}
     \Phi(N+1,x_0,\mathbf{u}_1)
\in & \ri\!\left(\ch(\Phi(1,x_0,\mathbf{u}_1),\dots,\Phi(N,x_0,\mathbf{u}_1))\right) \\
         \Phi(N+1,x_0,\mathbf{u}_1)
\in  & \Int(H).
\end{array}
\]
Iterating this construction, one shows by induction that for all $t \ge 1$, there exists $\mathbf{u}_t$ such that
\[\begin{array}{rl}
\Phi(t+N,x_0,\mathbf{u}_t) \in &  \Int(H),
\\
\Phi(t+N,x_0,\mathbf{u}_t)
\in & \ri\!\big(\ch(\Phi(t,x_0,\mathbf{u}_t),\dots,\\
&\Phi(t+N-1,x_0,\mathbf{u}_t))\big).
\end{array}
\]

\medskip

\noindent\emph{Step 3: Construction of an invariant set.}
For $t = N$, define
\[
S := \ch\big(\{\Phi(i,x_0,\mathbf{u}_N) \mid N \le i \le 2N-1\}\big).
\]

Since $\Phi(i,x_0,\mathbf{u}_N) \in \Int(H)$ for all $N \le i \le 2N-1$, we have
\[
S \subseteq \Int(H) \subseteq H.
\]

Using the same argument as in Theorem~\ref{prop:feas_inv}, $S$ is controlled invariant.

\medskip

\noindent\emph{Step 4: Convergence argument.}
Let $S_k$ denote the $k$-step backward reachable sets of $S$. Then
\[
S \subseteq \Int(H) \subseteq \Int(S_N).
\]

By \cite[Theorem~1]{liu2023convergencebackwardreachablesets}, the sequence $(S_k)$ converges exponentially fast (in Hausdorff distance) to the maximal RCI set $X_{\max}$.

Since $S \subseteq H \subseteq S_N$, the same convergence property holds for $H$.

\end{itemize}

\textcolor{blue}{\textbf{\underline{Proof of Lemma \ref{prop:rci_struc}:}}} \\
    Firstly, we show that $\cl(S)$ is an invariant. This can be done by approaching any element of the closure by a sequence of elements of $S$.  
    Since $X$ is closed, $\cl(S) \subseteq X$.  Consider $L>0$ such that $\|A\|\leq L$ and $\|B\|\leq L$.  Consider any $x\in \cl(S)$, then there exists $(x_t)_{t \in \mathbb{N}_{\geq 0}}$ such that $x_t \underset{S}{\rightarrow} x$. Since $S$ is a  controlled invariant, using Definition  \ref{def:RCI},  for all $t \geq 0$, we have the existence of  $u_t \in U$ such that , \begin{equation}
            \label{eqn:seq_rci}
                A x_t + B u_t ~ \in S. 
            \end{equation} 
            Since $U$ is compact, the sequence $u_t$  admit a convergent subsequence. Hence, there exists $u \in U$ and a strictly increasing sequence $t_k \in \mathbb{N}$ such that $u_{t_k} \rightarrow u$ . In the  following  we relabel $u_{t_k}$ to  $u_t$.  Let $ x_t^+ = A x_t + B u_t$ and  $x^+ = A x + B u$. We then have for all $t \in \mathbb{N}$: 
            \begin{equation}
                \begin{array}{rl}
                    \|x^+ - x_t^+\| =&  \|A x + B u - \\ 
                    & (A x_t + B u_t)\|\\
                    \leq  & \|A (x - x_t) + B (u-u_t)\| \\
                    \leq & \|A (x - x_t) \| + \|B (u-u_t)\|\\
                    \leq  & L(\|x - x_t\|+\|u-u_t\|).
                \end{array}
            \end{equation} where the second inequality comes from the application of the triangular inequality and  the third inequality comes from the definition of matrix norm. Since $\|x - x_t\|+\|u-u_t\| \rightarrow 0$, we have $\|x_t^+ - x^+\| \rightarrow 0 $. Finally, using (\ref{eqn:seq_rci}), one gets that  $x^+=A x + B u\in \cl(S)$. Then $\cl(S)$ is also a  controlled invariant.\\        
    Secondly, we show that $\ch(S)$ is a  controlled invariant. Since $X$ is convex, $\ch(S) \subseteq X$.  Consider any $x\in \ch(S)$. Using Carathéodory theorem \cite{caratheodory_uber_1911}, there exists $x_1, x_2, \dots, x_{n+1} \in S$ and $\alpha_1, \alpha_2, \dots , \alpha_{n+1} \geq 0$ such that $ \sum\limits_{i=1}^{n+1} \alpha_i = 1 $  and  $\sum\limits_{i=1}^{n+1} \alpha_i x_i= x $. Since $S$ is a  controlled invariant, 
     there exists $u_1, u_2, \dots, u _{n+1}$ such that $A x_i + B u_i ~ \in S$. Let $u = \sum\limits_{i=1}^{n+1} \alpha_i u_i$. We have:
     \begin{align*}
        A x + B u  &=A\sum\limits_{i=1}^{n+1} \alpha_i x_i + B \sum\limits_{i=1}^{n+1} \alpha_i u_i\\ =& \sum\limits_{i=1}^{n+1} \alpha_i (A x_i + B u_i)  \in ch(S). 
     \end{align*}
     where the inclusion follows directly from the definition of the convex hull.
Hence, $\ch(S)$ is a  controlled invariant.

Finally, using the previous results, $\cl(\ch(S)) \subseteq X$ is also a  controlled invariant.
\bigskip

\textcolor{blue}{\textbf{\underline{Proof of Lemma \ref{thm1}:}}}

\textbf{Sufficient condition:} Since S is a  controlled invariant for constraint set $(X, U)$, then for $x \in \ex(S) \subseteq S$, there exists $u \in U$ such that for, $A x + B u \in S$. Hence condition (\ref{eqn:prop_charact}) holds.

 \textbf{Necessary condition:} Suppose that condition (\ref{eqn:prop_charact}) holds. Let us show that condition (\ref{eqn:rci_charac}) holds. Since $S$ is a compact convex set, from the Krein-Millman theorem \cite{rudin_functional_1991}, $S = \cl(\ch(\ex(S))$. Consider any $x \in S$. Using the Carathéodory theorem, we have the existence of $\alpha_1, \dots,\alpha_{n+1}$,  $x_1, \dots,$ $ x_{n+1}$ such that for all $1 \leq i \leq n+1$, $0\leq \alpha_i \leq 1$, $x_i \in \ex(S)$, $\sum\limits_{i=1}^{n+1} \alpha_i =1$  and $x = \sum\limits_{i=1}^{n+1} \alpha_i x_i$. Since condition (\ref{eqn:prop_charact}) holds, we have the existence of $u_1, \dots, u_{n+1}$ such that, for all $0\leq i \leq n+1$
    \begin{equation}
        \label{eqn:ex_rci}
        Ax_i + B u_i \in S.
    \end{equation} Let $u = \sum\limits_{i=1}^{n+1} \alpha_i u_i \in U$. We have  
    \begin{equation}
        \begin{array}{rl}
             x^+ =& Ax + B u   \\
             =& \sum\limits_{i=1}^{n+1}\alpha_i\left[A  x_i + B  u_i \right] \, \in S
        \end{array}
    \end{equation}
    where the last line uses (\ref{eqn:ex_rci}) and the fact that $S$ is a convex set. Hence we conclude that $S$ is a  controlled invariant for constraint set $(X,U)$.
\medskip

\textcolor{blue}{\textbf{\underline{Proof of Corollary \ref{cor:1}:}}}
Let $x \in \ri(\ch(x_0, \dots, x_n))$.  To show equality of the affine closure  it's sufficient  to show that  $x_0 \in \af(\{x, x_1, \dots, x_n\})$.  Since  $x \in \ri(\ch(x_0, \dots, x_n))$, using Lemma \ref{lem:Int_char}, we can write  $x = \sum\limits_{i=0}^{n} \alpha_i x_i$ with $\alpha_i>0$ and $\sum\limits_{i=0}^{n} \alpha_i = 1$. Substituting  $x_0$  we get: 
$$x_0 = \frac{x}{\alpha_0} - \sum\limits_{i=1}^{n} \frac{\alpha_i}{\alpha_0} x_i.$$ Using the fact that  $ \frac{1}{\alpha_0} - \sum\limits_{i=1}^{n} \frac{\alpha_i}{\alpha_0} = \frac{1 - \sum\limits_{i=1}^{n} \alpha_i}{\alpha_0} =1$ ,  One gets  that $x_0 \in \af(\{x, x_1, \dots, x_n\})$. 
\fi 
\medskip
\ifitsdraft

\textcolor{blue}{\textbf{\underline{Proof of Theorem \ref{pro:mpc}:}}}

    Suppose that problem~\eqref{eqn:MPC_feas} is feasible at time $t$. Then there exists an optimal input sequence
\[
\mathbf{u}_t^N = \{u^*_{0\mid t}, \dots, u^*_{N-1\mid t}\}
\]
with associated state trajectory
\[
\mathbf{x}_t^N = \{x_{0\mid t}, \dots, x_{N\mid t}\}
\]
such that
\[
\begin{aligned}
& x_{k \mid t} \in X,\quad u_{k \mid t} \in U, \quad k = 0,\dots,N-1, \\
& x_{N \mid t} \in \Int\!\big(\ch(x_{0\mid t}, \dots, x_{N-1\mid t})\big).
\end{aligned}
\]

Applying $u^*_{0\mid t}$ yields
\[
x_{t+1} = x_{1\mid t} = A x_t + B u^*_{0\mid t}.
\]

We show feasibility at time $t+1$ by constructing a feasible candidate.

\medskip

Since 
\[
x_{N \mid t} \in \Int\!\big(\ch(x_{0\mid t}, \dots, x_{N-1\mid t})\big),
\]
there exist coefficients $\lambda_i > 0$, $i=0,\dots,N-1$, such that
\[
\sum_{i=0}^{N-1} \lambda_i = 1,
\qquad
x_{N \mid t} = \sum_{i=0}^{N-1} \lambda_i x_{i\mid t}.
\]

Define the shifted input sequence
\[
\mathbf{u}_{t+1}^N = \{u_{0\mid t+1}, \dots, u_{N-1\mid t+1}\}
\]
by
\[
u_{i\mid t+1} =
\begin{cases}
u^*_{i+1\mid t}, & 0 \le i \le N-2, \\
\sum_{i=0}^{N-1} \lambda_i u^*_{i\mid t}, & i = N-1.
\end{cases}
\]

Let $\mathbf{x}_{t+1}^N = \{x_{0\mid t+1}, \dots, x_{N\mid t+1}\}$ be the associated trajectory. Then
\[
x_{i\mid t+1} = x_{i+1\mid t}, \quad 0 \le i \le N-1,
\]
and
\[
\begin{aligned}
x_{N\mid t+1}
&= A x_{N-1\mid t+1} + B u_{N-1\mid t+1} \\
&= A x_{N\mid t} + B \sum_{i=0}^{N-1} \lambda_i u^*_{i\mid t} \\
&= \sum_{i=0}^{N-1} \lambda_i (A x_{i\mid t} + B u^*_{i\mid t}) \\
&= \sum_{i=0}^{N-1} \lambda_i x_{i+1\mid t} = \sum_{i=0}^{N-1} \lambda_i x_{i\mid t+1}.
\end{aligned}
\]

Thus,
\[
x_{N\mid t+1} \in \ch(x_{0\mid t+1}, \dots, x_{N-1\mid t+1}).
\]

Moreover, since $\lambda_i > 0$, Lemma~\ref{lem:Int_char} and Corollary \ref{cor:1} implies
\[
x_{N\mid t+1} \in \Int\!\big(\ch(x_{0\mid t+1}, \dots, x_{N-1\mid t+1})\big).
\]

Finally, since $\ch(x_{0\mid t}, \dots, x_{N-1\mid t}) \subseteq X$, we have
\[
x_{i\mid t+1} \in X, \quad i = 0,\dots,N-1,
\]
and clearly $u_{i\mid t+1} \in U$.

\medskip

Therefore, $\mathbf{u}_{t+1}^N$ is a feasible input sequence at time $t+1$, and problem~\eqref{eqn:MPC_feas} is recursively feasible.

\fi


\begin{thebibliography}{10}

\bibitem{blanchini_set-theoretic_2008}
F.~Blanchini and S.~Miani, {\em Set-theoretic methods in control}.
\newblock -: Springer, 2008.

\bibitem{kolmanovsky_gilbert_1998}
I.~Kolmanovsky and E.~G. Gilbert, ``Theory and computation of disturbance invariant sets for discrete-time linear systems,'' {\em Mathematical Problems in Engineering}, vol.~4, pp.~317--367, 1998.

\bibitem{saoud2024temporallogicresiliencedynamical}
A.~Saoud, P.~Jagtap, and S.~Soudjani, ``Temporal logic resilience for dynamical systems,'' {\em IEEE Transactions on Automatic Control}, 2026.

\bibitem{liu2023convergencebackwardreachablesets}
Z.~Liu and N.~Ozay, ``On the convergence of the backward reachable sets of robust controlled invariant sets for discrete-time linear systems,'' in {\em Proc. IEEE 61st Conf. Decision and Control (CDC)}, pp.~4270--4275, 2022.

\bibitem{ipopt}
A.~Wächter and L.~T. Biegler, ``On the implementation of an interior-point filter line-search algorithm for large-scale nonlinear programming,'' {\em Mathematical Programming}, vol. 106, no. 1, pp. 25--57, 2006.

\bibitem{9781450370189}
T.~Anevlavis and P.~Tabuada, ``A simple hierarchy for computing controlled invariant sets,'' in {\em Proc. 23rd ACM Int. Conf. Hybrid Systems: Computation and Control (HSCC)}, 2020.

\bibitem{Chenyuan_tahir_2019}
C.~Liu, F.~Tahir, and I.~Jaimoukha, ``Full‐complexity polytopic robust control invariant sets for uncertain linear discrete‐time systems,'' {\em International Journal of Robust and Nonlinear Control}, vol.~29, 04 2019.

\bibitem{BRIAO20219741}
S.~L. Brião, E.~B. Castelan, E.~Camponogara, and J.~G. Ernesto, ``Output feedback design for discrete-time constrained systems subject to persistent disturbances via bilinear programming,'' {\em Journal of the Franklin Institute}, vol.~358, no.~18, pp.~9741--9770, 2021.

\bibitem{anevlavs-lui-ozay_2024}
T.~Anevlavis, Z.~Liu, N.~Ozay, and P.~Tabuada, ``Controlled invariant sets: Implicit closed-form representations and applications,'' {\em IEEE Transactions on Automatic Control}, vol.~69, no.~7, pp.~4506--4521, 2024.
\bibitem{7862214}
M.~Rungger and P.~Tabuada, ``Computing robust controlled invariant sets of linear systems,'' {\em IEEE Transactions on Automatic Control}, vol.~62, no.~7, pp.~3665--3670, 2017.

\bibitem{Tahir_jaimoukha_2015}
F.~Tahir and I.~M. Jaimoukha, ``Low-complexity polytopic invariant sets for linear systems subject to norm-bounded uncertainty,'' {\em IEEE Transactions on Automatic Control}, vol.~60, no.~5, pp.~1416--1421, 2015.

\bibitem{914717}
E.~Kerrigan and J.~Maciejowski, ``Invariant sets for constrained nonlinear discrete-time systems with application to feasibility in model predictive control,'' in {\em Proceedings of the 39th IEEE Conference on Decision and Control (Cat. No.00CH37187)}, vol.~5, pp.~4951--4956 vol.5, 2000.

\bibitem{hiriart-urruty_fundamentals_2012}
J.~Hiriart-Urruty and C.~Lemaréchal, {\em Fundamentals of {Convex} {Analysis}}.Grundlehren {Text} {Editions}, Berlin Heidelberg: Springer, 2012.

\bibitem{rakovic_parameterized_2010}
S.~V. Rakovic and M.~Baric, ``Parameterized robust control invariant sets for linear systems: {Theoretical} advances and computational remarks,'' {\em IEEE Transactions on Automatic Control}, vol.~55, no.~7, pp.~1599--1614, 2010.
\bibitem{caratheodory_uber_1911}
C.~Carathéodory, ``Über den variabilitätsbereich der fourier’schen konstanten von positiven harmonischen funktionen,'' {\em Rendiconti del Circolo Matematico di Palermo (1884-1940)}, vol.~32, pp.~193--217, Dec. 1911.

\bibitem{antsaklis2006linear}
P.~Antsaklis and A.~Michel, {\em Linear Systems}.
\newblock Boston: Birkh{\"a}user Boston, 2006.

\bibitem{rakovic_kerrigan_mayne_2005_invariant}
S.~Rakovic, E.~Kerrigan, K.~Kouramas, and D.~Mayne, ``Invariant approximations of the minimal robust positively invariant set,'' {\em IEEE Transactions on Automatic Control}, vol.~50, no.~3, pp.~406--410, 2005.

\bibitem{Camacho2007}
E.~F. Camacho and C.~Bordons, {\em Constrained Model Predictive Control}, pp.~177--216. London: Springer London, 2007.

\bibitem{mayne2000constrained}
D.~Q. Mayne, J.~B. Rawlings, C.~V. Rao, and P.~O. Scokaert, ``Constrained model predictive control: Stability and optimality,'' {\em Automatica}, vol.~36, no.~6, pp.~789--814, 2000.

\bibitem{mejari_data-driven_2023}
M.~Mejari, S.~K. Mulagaleti, and A.~Bemporad, ``Data-{Driven} {Synthesis} of {Configuration}-{Constrained} {Robust} {Invariant} {Sets} for {Linear} {Parameter}-{Varying} {Systems},'' Sept. 2023. arXiv:2309.06998 [cs, eess, math].
\bibitem{Lofberg2004}
J.~L{\"{o}}fberg, ``Yalmip : A toolbox for modeling and optimization in matlab,'' in {\em In Proceedings of the CACSD Conference}, (Taipei, Taiwan), 2004.

\bibitem{MPT3}
M.~Herceg, M.~Kvasnica, C.~Jones, and M.~Morari, ``{Multi-Parametric Toolbox 3.0},'' in {\em Proc.~of the European Control Conference}, (Z\"urich, Switzerland), pp.~502--510, July 17--19 2013.

\bibitem{ernesto2024control}
J.~G. Ernesto {\em et~al.}, {\em Control design for constrained LTI and LPV systems via polyhedral set invariance}.
\newblock PhD thesis, UNIVERSIDADE FEDERAL DE SANTA CATARINA, 2024.

\bibitem{coupled_tanks_2}
X.~Zhou, C.~Li, T.~Huang, and M.~Xiao, ``Fast gradient-based distributed optimisation approach for model predictive control and application in four-tank benchmark,'' {\em IET Control Theory \& Applications}, vol.~9, no.~10, pp.~1579--1586, 2015.


\bibitem{althoff_goran_girard_tech_reach}
M.~Althoff, G.~Frehse, and A.~Girard, ``Set propagation techniques for reachability analysis,'' {\em Annual Review of Control, Robotics, and Autonomous Systems}, vol.~4, no.~Volume 4, 2021, pp.~369--395, 2021.

\bibitem{rudin_functional_1991}
W.~Rudin, {\em Functional {Analysis}}.
\newblock International series in pure and applied mathematics, McGraw-Hill, 1991.

\bibitem{9781450315678} M. ~ Rungger, and M. ~Mazo, and P.~Tabuada,''Specification-guided controller synthesis for linear systems and safe linear-time temporal logic'', 2013 ACM HSCC conference, pp.~333–342, 2013.


\bibitem{1618830}
S.V. ~Rakovic,  and E.C.~Kerrigan,  and D.Q.~Mayne,  and J.~Lygeros, ''Reachability analysis of discrete-time systems with disturbances'', IEEE Transactions on Automatic Control, volume 51 , pp. ~ 546-561, 2006.

\bibitem{9992752}
A.~Saoud and M.~Arcak, ``Characterizations and Computation of Controlled Invariants for Monotone Dynamical Systems,'' in {\em Proc. IEEE 61st Conf. Decision and Control (CDC)}, pp.~4990--4995, 2022.

\bibitem{saoud_characterization_2024}
A.~Saoud and M.~Arcak, ``Characterization, verification and computation of robust controlled invariants for monotone dynamical systems,'' {\em Mathematics of Control, Signals, and Systems}, vol.~36, no.~1, pp.~71--100, Mar. 2024.

\bibitem{WAFOWEMBE2024135}
E.~J. Wafo~Wembe and A.~Saoud, ``On Robust Controlled Invariants for Continuous-time Monotone Systems,'' {\em IFAC-PapersOnLine}, vol.~58, no.~11, pp.~135--140, 2024. 
\newblock Note: 8th IFAC Conf. Analysis and Design of Hybrid Systems (ADHS).
\end{thebibliography}
\end{document}